\documentclass[12pt,reqno]{amsart}

\usepackage{eucal,color,graphicx,subfigure,mathrsfs}

\numberwithin{equation}{section}

\newtheorem{theorem}{Theorem}[section]
\newtheorem{lemma}[theorem]{Lemma}

\theoremstyle{definition}

\theoremstyle{remark}

\numberwithin{equation}{section}

\newcommand{\bx}{\boldsymbol{x}}
\newcommand{\bu}{\boldsymbol{u}}
\newcommand{\bv}{\boldsymbol{v}}
\newcommand{\ble}{\underline{\boldsymbol{e}}}
\newcommand{\blE}{\underline{\boldsymbol{E}}}
\newcommand{\bR}{\boldsymbol{R}}

\begin{document}

\title{Direct computation of stresses in linear elasticity}

\author{Weifeng Qiu}
\address{Department of Mathematics, City University of Hong Kong, 83 tat Chee Avenue, Kowloon, Hong kong}
\email{weifeqiu@cityu.edu.hk}

\author{Minglei Wang}
\address{Shanghai Starriver bilingual school, China}
\email{minglei.wang@shbs.sh.cn}

\author{Jiahao Zhang}
\address{Department of Mathematics, City University of Hong Kong, 83 tat Chee Avenue, Kowloon, Hong kong}
\email{jiahzhang2-c@my.cityu.edu.hk}
\thanks{Corresponding author: Jiahao Zhang (jiahzhang2-c@my.cityu.edu.hk)}

\thanks{ {\bf Acknowledgements}. Authors would like to thank Professor Philippe G.~Ciarlet  for suggesting 
this topic. The work of Weifeng Qiu was supported by the GRF of Hong Kong (Grant No. 9041980).}




\begin{abstract}
We present a new finite element method based on the formulation introduced by Philippe G.~Ciarlet 
and Patrick Ciarlet, Jr. in [{\em Math. Models Methods Appl. Sci., 15 (2005), pp. 259--571}], 
which approximates strain tensor directly. We also show the convergence rate of strain tensor is optimal.
This work is a non-trivial generalization of its two dimensional analogue in 
[{\em Math. Models Methods Appl. Sci., 19 (2009), pp. 1043--1064}]
\end{abstract}

\maketitle

\section{introduction}

This is a continuation of the two-part article proposed by Ciarlet \& Ciarlet, Jr. 
\cite{Ciarlet:2005:strain,Ciarlet:2009:strain}. The main objective of this article is 
to introduce and analyze a direct finite element approximation of the minimization problem 
$j(\underline{\boldsymbol{\varepsilon}})=\inf\limits_{\ble \in \blE(\Omega)}j(\ble)$, which will be introduced 
in $(\ref{new_mini_prob})$, to compute stresses inside an elastic body precisely in three dimensional space. 
The notations and ideas used in this paper to explain our finite element method are due to
Ciarlet \& Ciarlet, Jr. (See \cite{Ciarlet:2009:strain}).

Let $\mathbb{S}^3$ denote the space of all $3\times3$ symmetric matrix. Let $\Omega$ be an open, bounded, connected subset of $\mathbb{R}^3$ with Lipschitz boundary. Let $\underline{\boldsymbol{a}}:\underline{\boldsymbol{b}}$ denote the inner product of two matrices $\underline{\boldsymbol{a}}$ and $\underline{\boldsymbol{b}}$. Now consider a homogeneous, isotropic linearly elastic body with Lame's constants $\lambda>0$ and $\mu>0$, with $\overline{\Omega}$ as its reference configuration, and subjected to applied body forces, of density $\boldsymbol{f}\in L^{\frac{6}{5}}(\Omega;\mathbb{R}^{3})$ in its interior and of 
density $\boldsymbol{g}\in L^{\frac{4}{3}}(\Omega;\mathbb{R}^{3})$ on its boundary $\Gamma$. 
Given any matrix $\ble=(e_{ij})\in{\mathbb{S}^3}$, $A \ble\in{\mathbb{S}^3}$ is defined by 
$$A \ble=\lambda(\text{tr} \ble)\mathbf{I}_{3}+2\mu \ble.$$ 
The classical way to solve pure traction problem of three-dimensional linearized elasticity is to find a displacement vector field 
$\bu\in H^1(\Omega;\mathbb{R}^{3})$ which satisfies $$J(\bu)=\inf_{\bv\in H^1(\Omega;\mathbb{R}^{3})}J(\bv),$$ with
\begin{equation}
J(\bv)=\frac{1}{2}\int_{\Omega}A\bigtriangledown_s \bv:\bigtriangledown_s \bv dx-L(\bv).
\end{equation}
For all $\bv\in H^1(\Omega;\mathbb{R}^{3})$, where $$L(\bv)=\int_{\Omega}\boldsymbol{f}\cdot \bv dx
+\int_{\Gamma}\boldsymbol{g}\cdot \bv d\Gamma,$$ and 
$$\bigtriangledown_s \bv=\frac{1}{2}(\bigtriangledown \bv^{\top}+\bigtriangledown \bv)\in{L^2(\Omega;\mathbb{S}^3)},$$ 
denotes the linearized strain tensor field associated with any vector field $\bv\in H^1(\Omega;\mathbb{R}^{3})$.

If the applied body forces hold for the compatibility condition $L(\bv)=0$ for all $\bv\in \bR(\Omega)$, where 
$$\bR(\Omega)=\{\bv\in H^1(\Omega;\mathbb{R}^{3}): \bigtriangledown_s \bv=0\text{ in }\Omega \}
=\{\bv=\boldsymbol{a}+\boldsymbol{b}\wedge \boldsymbol{x}:\boldsymbol{a},\boldsymbol{b}\in \mathbb{R}^{3}\}.$$
Then $\inf_{\bv\in H^1(\Omega; \mathbb{R}^{3})} J(\bv)>-\infty$. Besides, thanks to the $Korn's\text{ }inequality$, 
the solutions of the minimization problem above exist (See \cite{Duvaut & Lions :1972:strain}) and 
they are unique up to the addition of any vector field $\bv\in{R(\Omega)}$.

Instead of finding vector field $\bu$ directly, Ciarlet \& Ciarlet,Jr (See \cite{Ciarlet:2005:strain}) put forward 
a new method to define a new unknown $\ble$ which is a $d\times d$ symmetric matrix field in $d$-dimension ($d=2,3$). 
It is proved by Ciarlet \& Ciarlet,Jr in \cite{Ciarlet:2005:strain} that in both two and three dimensions (note that, Antman 
\cite{S. S. Antman:1976:strain} already proposed similar idea without proof in 1976 in three-dimensional nonlinear elasticity, 
while ours is concerned about linear elasticity), if $\ble\in \blE(\Omega)$ where 
$\blE(\Omega)=\{\ble|\textbf{ curl }\textbf{curl } \ble=0; \ble\in{L_s}^2(\Omega)\}$, 
then any vector field satisfying $\bigtriangledown_s \bv=\ble$ lies in the set 
$\{\bv|\bv=\dot{\bv}+\boldsymbol{r}\text{ for some } \boldsymbol{r}\in \bR(\Omega)\}$ and the mapping $\kappa:\blE(\Omega)\rightarrow\dot{\bv}\in\dot{H^1}(\Omega;\mathbb{R}^{3})=H^1(\Omega;\mathbb{R}^{3})/\bR(\Omega)$ 
such that $\bigtriangledown_{s}\dot{\bv}=\ble$ is an isomorphism between $\dot{H^1}(\Omega;\mathbb{R}^{3})$ 
and $\blE(\Omega)$ (See \cite{Ciarlet:2007:strain, Ciarlet:20072:strain, Ciarlet:2006:strain, G.Geymonat:2005:strain}). 
Since the relationship between $\ble$ and $\bv$ is clear (we choose $\ble = \bigtriangledown_s \bv$), the minimization 
problem is converted to: 
$$j(\underline{\boldsymbol{\varepsilon}})=\inf_{\ble\in \blE(\Omega)}j(\ble)$$
\begin{equation}
j(\ble)=\frac{1}{2}\int_{\Omega}A\ble:\ble dx-l(\ble),
\end{equation}
for all $\ble\in \blE(\Omega)$, where$$l(\ble)=L\circ\kappa,$$and we can further show that 
$\underline{\boldsymbol{\varepsilon}}=\bigtriangledown_s \bu$.

Let $\Omega$ be a triangulation of a polyhedral domain in $\mathbb{R}^{3}$ and $\blE^h$ be a finite element subspace of $\blE(\Omega)$. 
We consider $\ble^h\in \blE^h$ to be $3\times3$ symmetric matrix with piecewise constant element. Our goal is to illustrate 
that how $\ble^h$ can satisfy the condition
\begin{equation}
\textbf{curl } \textbf{curl } \ble^h=0
\end{equation}

Note that the similar ideas of the mixed finite element methods for linearized elasticity have been discussed in \cite{D.N.Arnold:1984:strain,D.N.Arnold:19842:strain,D.N.Arnold:1988:strain,D.N.Arnold:2002:strain,D.N.Arnold:2003:strain,D.N.Arnold:2005:strain}.

In order to achieve $(1.3)$, the six degrees of freedom that define the elements $\ble^h\in \blE^h$must be supported by edges 
(See Lemma$~\ref{lemma_dof}$) and they must also satisfy specific compatibility conditions (See Theorem$~\ref{def_spaces}, \ref{interior_boundary_vertex}$ and $\ref{computational_space}$). Then the minimization problem can be accomplished. The associated finite elements thus provide examples of edge finite elements in the sense of Nedelec (See \cite{J.C:1980:strain,J.C:1986:strain}).

Finally, we consider discrete problem and it naturally comes to find a discrete matrix field 
$\underline{\boldsymbol{\varepsilon}}^h\in \blE^h$ such that $$j(\underline{\boldsymbol{\varepsilon}}^h)
=\inf_{\ble^h\in \blE^h}j(\ble^h).$$ has unique solution (See Theorem$~\ref{dis_mini_pro}$), and we establish 
convergence of the method (See Theorem$~\ref{convergence_pro}$). 

Note that this approach is in a sense the "matrix-analog" of the approximation of the Stokes problem by means of the divergence-free finite elements of Crouzeix \& Raviart (although theirs are non-conforming, whereas ours are not).

In the following part of the paper, the details of the approach and its validation will be discussed in section 2. Then the computational space is constructed in section 3. As it showed above, the final section will be the discussion of discrete problem and the convergence of the method.

\section{Three dimensional linearized elasticity}

The basic notations and conceptions are introduced in section two of [1]. Let $x_{i}$ denote the coordinates of a point $\bx\in R^3$, let $\partial_{i}:=\partial /\partial x_{i}$ and $\partial_{ij}:=\partial^2/\partial x_{i}\partial x_{j}$, $i,j\in \{1,2,3\}$. Given a smooth enough vector field $\bv$, we define the $3\times3$ matrix field $\bigtriangledown \bv:=(\partial_{i} v_{j})$. In this paper, we first consider the domain which is open, bounded and connected subset of $\mathbb{R}^3$ with Lipschitz-continuous boundary. Let $\Omega$ be a  domain in $\mathbb{R}^3$. Given any vector field $\bv\in H^1(\Omega; \mathbb{R}^3)$ as a displacement field, define
\begin{equation}
\bigtriangledown_s \bv:=\frac{1}{2}(\bigtriangledown \bv^{\top}+\bigtriangledown \bv)
\in L_s^2(\Omega):=L^2(\Omega;\mathbb{S}^3),
\end{equation}
as its associated symmetrized gradient matrix field. Let
\begin{equation}
\label{R_space}
\bR(\Omega):=\{\boldsymbol{r}\in H^1(\Omega;\mathbb{R}^{3});\bigtriangledown_s \boldsymbol{r}=0\}
\end{equation}
denote the space of infinitesimal rigid displacement fields. It is obvious that 
$\boldsymbol{r}\in \bR(\Omega)$ if and only if there exists a constant vector $\boldsymbol{c}\in\mathbb{R}^{3}$ 
and a matrix $B$ satisfying $B^{\top}=-B$ such that $$\boldsymbol{r}=B\boldsymbol{x}+\boldsymbol{c},$$ 
for all $\boldsymbol{x}=(x_1,x_2,x_3)\in\Omega$.

It is discussed in \cite{Ciarlet:2009:strain} that the associated pure traction minimization problem
\begin{equation}
\label{classical_mini_prob}
J(\bu)=\inf_{\bv\in H^1(\Omega;\mathbb{R}^{3})}J(\bv), \text{ where } J(\bv)=
\frac{1}{2}\int_{\Omega}A\bigtriangledown_s \bv:\bigtriangledown_s \bv dx-L(\bv).
\end{equation}
with $$L(\bv)=\int_{\Omega}\boldsymbol{f}\cdot \bv dx+\int_{\Gamma}\boldsymbol{g}\cdot \bv d\Gamma,$$ 
has one and only one solution if it satisfies the compatibility condition $L(\boldsymbol{r})=0$ 
for all $\boldsymbol{r}\in \bR(\Omega)$, which can be supported by Korn's inequality.

Instead of seeking the displacement field $\bu$ in the classical approach, the intrinsic approach views the linearized 
strain tensor field $\underline{\boldsymbol{\varepsilon}}:=\bigtriangledown_s \bu$ directly as the primary unknown. 
This changing of variables can be accomplished with the following results (See Theorem $2.1$,$2.2$ and $2.3$), 
which were proved in \cite{Ciarlet:2005:strain} by Ciarlet and Ciarlet,Jr.

\begin{theorem}
Let $\Omega$ be a simply-connected domain in $\mathbb{R}^3$ and let $\ble=(e_{ij})\in L_s^2(\Omega)$ be 
a tensor field that satisfies $$\textbf{curl } \textbf{curl }\ble=0\text{ in }H^{-2}(\Omega;\mathbb{R}^{3\times 3}).$$ 
Then there exists a vector field $\bv\in H^1(\Omega;\mathbb{R}^{3})$ such that $\bigtriangledown_s \bv= \ble$ 
in $L_s^2(\Omega)$, and all the other solutions $\overline{\bv}$ 
are of the form $$\overline{\bv}=\bv+\boldsymbol{r}\text{ for some } \boldsymbol{r}\in \bR(\Omega),$$ 
where $\bR(\Omega)$ is the space defined in $(\ref{R_space})$.
\end{theorem}

\begin{theorem}
Let $\Omega$ be a simply-connected domain in $\mathbb{R}^3$. Define the space
\begin{equation}
\blE(\Omega):=\{\ble\in L_s^2(\Omega);\textbf{curl } \textbf{curl } \ble=0\text{ in }H^{-2}(\Omega;\mathbb{R}^{3\times 3})\},
\end{equation}
and given any $\ble\in \blE(\Omega)$, let $\dot{\bv}=\kappa(\ble)$ denote the unique element in the space 
$\dot{H}^1(\Omega;\mathbb{R}^{3})$ that satisfies $\ble=\bigtriangledown_s\dot{\bv}$ (See Theorem $2.1$). 
Then the linear mapping $$\kappa:\blE(\Omega)\rightarrow\dot{H}^1(\Omega;\mathbb{R}^{3})$$ is an isomorphism 
between the Hilbert space $\blE(\Omega)$ and $\dot{H}^1(\Omega;\mathbb{R}^{3})$.
\end{theorem}

\begin{theorem}
Let $\Omega$ be a simply-connected domain in $\mathbb{R}^3$. Then the minimization problem: 
Find $\underline{\boldsymbol{\varepsilon}}\in \blE(\Omega)$ such that
\begin{equation}
\label{new_mini_prob}
j(\underline{\boldsymbol{\varepsilon}})=\inf_{\ble\in \blE(\Omega)}j(\ble), \text{ where }j(\ble)
=\frac{1}{2}\int_{\Omega}A\ble:\ble dx-l(\ble),
\end{equation}
with $$l(\ble)=L\circ\kappa,$$ has one and only one solution $\underline{\boldsymbol{\varepsilon}}$ 
and $$\underline{\boldsymbol{\varepsilon}}:=\bigtriangledown_s \bu$$ where $\bu$ is the unique solution 
to problem $(\ref{classical_mini_prob})$.
\end{theorem}

\section{Finite element space for strain tensor}

Same as Ciarlet \& Ciarlet, Jr. did in \cite{Ciarlet:2009:strain}, we first describe a triangular finite element, 
which provides edge finite element. The length element is denote $dl$.

\begin{lemma}
\label{lemma_dof}
Let $T$ be a non-degenerate tetrahedron with edges $s_{i},1\leq i\leq 6$. Given any edge $s_{i}$ of $T$,
let $\tau^{i}$ denote a unit vector parallel to $s_{i}$, and let the degrees of freedom $d_{i},1\leq i\leq 6$,
be defined as 
\begin{equation}
\label{dof}
d_{i}(\ble):=\int_{s_{i}}\tau^{i}\cdot \ble\tau^{i}dl\qquad \forall \ble\in P_{0}(T;\mathbb{S}^{3}).
\end{equation}
Then, the set $\{d_{i};1\leq i\leq 6\}$ is $P_{0}(T;\mathbb{S}^{3})$-unisolvent, i.e. a tensor field $\ble
\in P_{0}(T;\mathbb{S}^{3})$ is uniquely determined by the six numbers $d_{i}(\ble),1\leq i\leq 6$.
\end{lemma}

\begin{proof}
The proof will be the same as that of Theorem $3.1$ in \cite{Ciarlet:2009:strain}.
\end{proof}

From now on, $\Omega$ denotes a polynomial domain in $\mathbb{R}^{3}$, and we consider triangulations $\mathcal{T}^{h}$ 
of the set $\bar{\Omega}$ by tetrahedrons $T\in\mathcal{T}^{h}$ subjected to the usual conditions; in particular, all 
the tetrahedrons $T\in\mathcal{T}^{h}$ are non-degenerate.

Given such a triangulation $\mathcal{T}^{h}$ of $\bar{\Omega}$, let $\Sigma^{h}$ denote the set of all ``interior" edges found 
in $\mathcal{T}^{h}$ (i.e. that are not contained in the boundary $\partial\Omega$), let $\Sigma_{\partial}^{h}$ denote the set 
of all ``boundary" edges found in $\mathcal{T}^{h}$ (1.e. that are contained in $\partial\Omega$), let $A^{h}$ denote the 
set of all ``interior" vertices in $\mathcal{T}^{h}$ (i.e. that are contained in $\Omega$), and let $A_{\partial}^{h}$ denote the 
set of all ``boundary" vertices in $\mathcal{T}^{h}$ (i.e. that are contained in $\partial\Omega$).

We also assume that each interior or boundary edge $\sigma\in \Sigma^{h}\cup \Sigma_{\partial}^{h}$ is oriented. 
For instance, if $\{\boldsymbol{a}^{j};1\leq j\leq J\}$ denotes the set of all the vertices found in $\mathcal{T}^{h}$ 
and $\sigma=[\boldsymbol{a}^{i},\boldsymbol{a}^{j}]\in\Sigma^{h}\cup \Sigma_{\partial}^{h}$ with $i<j$, 
one may let $\boldsymbol{\tau}:\vert \boldsymbol{a}^{j}-\boldsymbol{a}^{i} \vert^{-1}
(\boldsymbol{a}^{j}-\boldsymbol{a}^{i})$.

\begin{theorem}
\label{def_spaces}
Given any triangulation $\mathcal{T}^{h}$ of $\bar{\Omega}$, define the finite element space 
\begin{align}
\label{source_space}
\tilde{\underline{\mathbb{E}}}^{h}:=\lbrace & \ble^{h}\in\mathbb{L}_{s}^{2}(\Omega);\ble^{h}|_{T}
\in P_{0}(T;\mathbb{S}^{3}) \text{ for all }T\in\mathcal{T}^{h}, \text{ and for all }\sigma\in\Sigma^{h}, \\
\nonumber  
& \int_{\sigma}\boldsymbol{\tau}\cdot (\bold{e}^{h}|_{T})\boldsymbol{\tau} dl
\text{ is single-valued for all } T\in\mathcal{T}^{h}\text{ with } \sigma\subset T
\rbrace .
\end{align}
Then each tensor field $\ble^{h}\in\tilde{\underline{\mathbb{E}}}^{h}$ is uniquely defined by the numbers 
$d_{\sigma}(\ble^{h}),\sigma\in\Sigma^{h}\cup\Sigma_{\partial}^{h}$ where the degrees of freedom 
$d_{\sigma}:\tilde{\underline{\mathbb{E}}}^{h}\rightarrow\mathbb{R}$ are defined by 
\begin{equation}
\label{global_edge_dof}
d_{\sigma}(\ble^{h}) = \int_{\sigma}\boldsymbol{\tau}\cdot (\ble^{h}|_{T})\boldsymbol{\tau} dl 
\text{ for any }T\in\mathcal{T}^{h}\text{ with } \sigma\subset T.
\end{equation}
 
Define the finite element space 
\begin{equation}
\label{exact_space}
\hat{\underline{\mathbb{E}}}^{h}:=\{\bigtriangledown_{s}\dot{\bv}^{h}\in\mathbb{L}_{s}^{2}(\Omega);\dot{\bv}^{h}
\in\dot{\boldsymbol{V}}^{h}\},
\end{equation} 
where, the space $\bR(\Omega):=\{\bv=\boldsymbol{a}+\boldsymbol{b}
\wedge\boldsymbol{x}; \boldsymbol{a},\boldsymbol{b}\in\mathbb{R}^{3}\}$,
\begin{equation}
\label{V_space}
\dot{\boldsymbol{V}}^{h}:=\boldsymbol{V}^{h}/\bR(\Omega)\text{ with }
\boldsymbol{V}^{h}:=\{\bv^{h}\in\boldsymbol{C}^{0}(\bar{\Omega});
\bv^{h}|_{T}\in P_{1}(T;\mathbb{R}^{3})\}.
\end{equation}
Then,
\begin{equation}
\label{space_inclusion}
\hat{\underline{\mathbb{E}}}^{h}\subset \tilde{\underline{\mathbb{E}}}^{h}.
\end{equation}
Besides,
\begin{equation}
\label{dim_exact_space}
\dim \hat{\underline{\mathbb{E}}}^{h} = \dim \boldsymbol{V}^{h} - 6.
\end{equation}
\end{theorem}

\begin{proof}
According to Lemma~\ref{lemma_dof}, each tensor field $\ble^{h}\in\tilde{\underline{\mathbb{E}}}^{h}$ 
is uniquely defined by the numbers $d_{\sigma}(\ble^{h})$ in (\ref{global_edge_dof}).  
The proof of (\ref{space_inclusion},\ref{dim_exact_space}) will be 
similar to that of Lemma $3.1$ in \cite{Ciarlet:2009:strain}.
\end{proof}

\begin{lemma}
\label{interior_boundary_vertex}
Given any triangulation $\mathcal{T}^{h}$ of $\bar{\Omega}$, for any vertex $\boldsymbol{a}$ found in $\mathcal{T}^{h}$,
let $\{T;T\in\mathcal{T}^{h}(a)\}$ denote the set formed by all the tetrahedrons of $\mathcal{T}^{h}$ that have 
the vertex $\boldsymbol{a}$ in common, and let $\bar{\Omega}_{\boldsymbol{a}}:=\overline{\bigcup_{T\in\mathcal{T}^{h}}T}$.

For any given vertex $\boldsymbol{a}$ found in $\mathcal{T}^{h}$, We denote $N$ number of vertices 
in $\bar{\Omega}_{\boldsymbol{a}}$, $A$ number of edges in $\bar{\Omega}_{\boldsymbol{a}}$.
If $\boldsymbol{a}\in \Sigma^{h}$, then
\begin{equation}
\label{Euler_interior}
A = (3N-6)+(N_{b}-3),
\end{equation}
where, $N_{b}$ is number of vertices in the boundary of $\bar{\Omega}_{\boldsymbol{a}}$.
If $\boldsymbol{a}\in \Sigma_{\partial}^{h}$, then
\begin{equation}
\label{Euler_boundary}
A = (3N-6)+N_{ib},
\end{equation}
where, $N_{ib}$ is number of vertices in $\partial\bar{\Omega}_{\boldsymbol{a}}\backslash\partial\Omega$.
\end{lemma}

\begin{proof}
If $\boldsymbol{a}\in\Sigma_{\partial}^{h}$, we denote $A_{ib}$ number of edges found in 
$\partial\bar{\Omega}_{\boldsymbol{a}}\backslash\partial\Omega$. Then,
\begin{equation*}
A = A_{ib}+2(N-N_{ib}-1)+N_{ib}.
\end{equation*}
And, we utilize ($3.12$) in \cite{Ciarlet:2009:strain} for $\partial\bar{\Omega}_{a}
\backslash\partial\Omega$, then 
\begin{equation*}
A_{ib}+(N-N_{ib}-1)=N_{ib}+2(N-1)-3.
\end{equation*}
So, we have that
\begin{align*}
A & = A_{ib}+2(N-N_{ib}-1)+N_{ib} \\
& = [N_{ib}+2(N-1)-(N-N_{ib})-2]+2(N-N_{ib}-1)+N_{ib}\\
& = (3N-6)+N_{ib}.
\end{align*}
This proves (\ref{Euler_boundary}).

If $\boldsymbol{a}\in\Sigma^{h}$, we take $T\in \mathcal{T}^{h}(\boldsymbol{a})$ arbitrarily.
Then, it is easy to see that $A$ and $N$ are the same number of edges 
and number of vertices in $\bar{\Omega}_{\boldsymbol{a}}\backslash T$, respectively.
We treat $\partial\bar{\Omega}_{\boldsymbol{a}}\backslash \partial T$ the same 
as $\partial\bar{\Omega}_{\boldsymbol{a}}\backslash\partial\Omega$ in above argument.
This will immediately prove (\ref{Euler_interior}).
\end{proof}

\begin{theorem}
\label{computational_space}
Given any vertex $\boldsymbol{a}\in A_{\partial}^{h}\cup A^{h}$, there exist $m$-many linearly independent 
linear forms $\varphi_{\boldsymbol{a},j}:\tilde{\underline{\mathbb{E}}}^{h}\rightarrow\mathbb{R}$ 
($1\leq j\leq m$) such that 
\begin{equation*}
\textbf{curl }\textbf{curl }\ble^{h} = 0 \quad \text{ in }\mathcal{D}^{\prime}(\Omega)\quad 
\text{ for all }\ble^{h}\in\underline{\mathbb{E}}^{h},
\end{equation*}
where the space $\underline{\mathbb{E}}^{h}$ is defined by 
\begin{align}
\label{def_computational_space}
\underline{\mathbb{E}}^{h}:=\{ & \ble^{h}\in \tilde{\underline{\mathbb{E}}}^{h}; 
\varphi_{\boldsymbol{a},j}(\ble^{h})=0, 1\leq j\leq m\text{ for all }
\boldsymbol{a}\in A_{\partial}^{h}\cup A^{h}\}.
\end{align}
Here, $m=N_{ib}$ if $\boldsymbol{a}\in A_{\partial}^{h}$, and $m=N_{b}-3$ if $\boldsymbol{a}\in A^{h}$.  
$N_{ib}$ and $N_{b}$ are defined in Lemma~\ref{interior_boundary_vertex}.

More specifically, the coefficients of each linear form $\varphi_{\boldsymbol{a},j}$ are 
explicitly computable functions of coordinates of the vertex $\boldsymbol{a}$ and the 
vertices of the tetrahedrons of $\mathcal{T}^{h}$ that have $\boldsymbol{a}$ as a vertex.

In addition, if the domain $\Omega$ is simply-connected, 
then $\underline{\mathbb{E}}^{h}=\hat{\underline{\mathbb{E}}}^{h}$.
\end{theorem}

\begin{proof}
Given any vertex $\boldsymbol{a}\in A_{\partial}^{h}\cup A^{h}$, we denote $\bar{\Omega}_{\boldsymbol{a}}$ 
the closure of union of all tetrahedrons which have vertex $\boldsymbol{a}$. It is easy to see that 
\begin{equation*}
\dim \tilde{\underline{\mathbb{E}}}^{h}|_{\bar{\Omega}_{\boldsymbol{a}}} 
= \text{number of edges in }\bar{\Omega}_{\boldsymbol{a}},\qquad 
\dim \hat{\underline{\mathbb{E}}}^{h}|_{\bar{\Omega}_{\boldsymbol{a}}} 
= 3\times(\text{number of vertices in }\bar{\Omega}_{\boldsymbol{a}})-6.
\end{equation*}
Then, by Lemma~\ref{interior_boundary_vertex},
\begin{equation*}
\dim \tilde{\underline{\mathbb{E}}}^{h}|_{\bar{\Omega}_{\boldsymbol{a}}}
 - \dim \hat{\underline{\mathbb{E}}}^{h}|_{\bar{\Omega}_{\boldsymbol{a}}} 
= N_{ib} \text{ if }\boldsymbol{a}\in A_{\partial}^{h},\qquad 
\dim \tilde{\underline{\mathbb{E}}}^{h}|_{\bar{\Omega}_{\boldsymbol{a}}}
 - \dim \hat{\underline{\mathbb{E}}}^{h}|_{\bar{\Omega}_{\boldsymbol{a}}} = N_{b}-3 
 \text{ if }\boldsymbol{a}\in A^{h}. 
\end{equation*}
This implies that there exist $m$-many linearly independent linear form 
$\varphi_{\boldsymbol{a},j}:\tilde{\underline{\mathbb{E}}}^{h}|_{\bar{\Omega}_{\boldsymbol{a}}}\rightarrow 
\mathbb{R}$ ($1\leq j\leq m$) such that $\hat{\underline{\mathbb{E}}}^{h}|_{\bar{\Omega}_{\boldsymbol{a}}}$ 
is equal to
\begin{equation}
\label{patch_computational_space}
\{\ble^{h}\in \tilde{\underline{\mathbb{E}}}^{h}|_{\bar{\Omega}_{\boldsymbol{a}}};
\varphi_{\boldsymbol{a},j}(\ble^{h})=0,1\leq j\leq m\}.
\end{equation}

Now, we take $\ble^{h}\in \tilde{\mathbb{E}}^{h}$ arbitrarily. 
If $\varphi_{\boldsymbol{a},j}(\bold{e}^{h}|_{\bar{\Omega}_{a}})=0$ for all 
$1\leq j\leq m$, then we have that $\ble^{h}|_{\bar{\Omega}_{\boldsymbol{a}}}\in 
\hat{\underline{\mathbb{E}}}^{h}|_{\bar{\Omega}_{\boldsymbol{a}}}$. This implies that
\begin{equation*}
\textbf{curl }\textbf{curl }\ble^{h} = 0\text{ in }\mathcal{D}^{\prime}(\text{int}(\bar{\Omega}_{\boldsymbol{a}})).
\end{equation*}
We notice that $\Omega = \bigcup_{\boldsymbol{a}\in A_{\partial}^{h}\cup A^{h}}\text{int}
(\bar{\Omega}_{\boldsymbol{a}})$. So, we can conclude that 
\begin{equation*}
\textbf{curl }\textbf{curl }\ble^{h} = 0\text{ in }\mathcal{D}^{\prime}(\Omega).
\end{equation*}
And, for any $\ble^{h}\in \tilde{\underline{\mathbb{E}}}^{h}$ and 
any $\boldsymbol{a}\in A_{\partial}^{h}\cup A^{h}$, $1\leq j\leq m$, we can define 
\begin{equation*}
\varphi_{\boldsymbol{a},j}(\ble^{h}):=\varphi_{\boldsymbol{a},j}(\ble^{h}|_{\bar{\Omega}_{\boldsymbol{a}}}).
\end{equation*}
So, we can conclude that
\begin{equation*}
\textbf{curl }\textbf{curl }\ble^{h} = 0 \quad \text{ in }\mathcal{D}^{\prime}(\Omega)\quad 
\text{ for all }\ble^{h}\in\underline{\mathbb{E}}^{h}.
\end{equation*}
Obviously, $\hat{\underline{\mathbb{E}}}^{h}\subset \underline{\mathbb{E}}^{h}$.

If the domain $\Omega$ is simply-connected, then using the same argument in the last paragraph of the proof of Theorem $3.3$, 
we an conclude that $\underline{\mathbb{E}}^{h}=\hat{\underline{\mathbb{E}}}^{h}$.
\end{proof}

\section{The discrete problem and convergence}
In the following, Let $\Omega$ be a simply-connected polygonal domain in $\mathbb{R}^3$. 
The discrete problem is defined as the minimization problem (See $(4.1)$) below.
\begin{theorem}
\label{dis_mini_pro}
Given any triangulation $\Gamma_h$ of $\overline{\Omega}$, let $\underline{\mathbb{E}}^h$ be 
the finite element space defined in ~(\ref{def_computational_space}). Then there exists one and 
only one $\underline{\boldsymbol{\varepsilon}}^h\in \underline{\mathbb{E}}^h$ such that
\begin{equation}
j(\underline{\boldsymbol{\varepsilon}}^h)=\inf_{\ble^h\in \underline{\mathbb{E}}^h}j(\ble^h),
\end{equation}
where $j$ is defined in $(\ref{new_mini_prob})$. Besides, let J be the function defined in $(\ref{classical_mini_prob})$, 
then $\underline{\boldsymbol{\varepsilon}}^h=\bigtriangledown_s\dot{\bu}^h$, where $\dot{\bu}^h$ is 
the unique solution to the minimization problem
\begin{equation}
J(\dot{\bu}^h)=\inf_{\cdot{\bv}^h\in \boldsymbol{V}^h}J(\dot{\bv}^h).
\end{equation}
\end{theorem}
\begin{proof}
The proof is the same as Theorem $4.1$ in \cite{Ciarlet:2009:strain}.
\end{proof}

Finally, we discuss the convergence of the method.
\begin{theorem}
\label{convergence_pro}
Consider a regular family of triangulation $\Gamma_h$ of $\overline{\Omega}$. Then
\begin{equation}
\Vert\boldsymbol{\underline{\varepsilon}}-\boldsymbol{\underline{\varepsilon}}^h\Vert_{L_s^2}
\rightarrow 0\text{ as }h\rightarrow 0.
\end{equation}
If $\bu\in H^2(\Omega;\mathbb{R}^{3})$, there exists a constant $C$ independent of $h$ such that
\begin{equation}
\Vert\boldsymbol{\underline{\varepsilon}}-\boldsymbol{\underline{\varepsilon}}^h\Vert_{L_s^2}
\leq C\Vert \bu\Vert_{H^2(\Omega)}h.
\end{equation}
\begin{proof}
The proof is the same as Theorem $4.2$ in \cite{Ciarlet:2009:strain}.
\end{proof}
\end{theorem}

\end{document}